\documentclass{article}

 \usepackage{amsmath,amssymb,latexsym, amsthm}
\usepackage{fullpage}
\usepackage{etoolbox}
\usepackage{enumerate}
\usepackage{color}
\usepackage[colorlinks,linkcolor=blue,citecolor=magenta, pagebackref]{hyperref}
\usepackage{soul}
\usepackage{microtype}
\usepackage{yhmath}
\usepackage{graphicx}
\usepackage{stackrel}
\usepackage{mathtools}
\usepackage{array}
\usepackage{float}
\usepackage{complexity}
\usepackage[sharp]{easylist}
\usepackage{tikz}
\usepackage{subcaption}
\usepackage[mathlines]{lineno}

\title{On Conflict-Free Colorings of Cyclic Polytopes and the Girth Conjecture for Graphs}

\author{Seunghun Lee\footnote{Seunghun Lee, Department of Mathematics, Keimyung University, Daegu, South Korea. Supported by the Institute for Basic Science (IBS-R029-C1) and the Department of Mathematical Sciences of Korea Advanced Institute of Science and Technology (BK21). \texttt{seunghun.math@gmail.com}. https://orcid.org/0000-0003-0838-1680} 
	\ and Shakhar Smorodinsky\footnote{
		Institute for the Theory of Computing, Ben-Gurion University of the NEGEV, Be'er Sheva 84105, Israel. Partially supported by Grant 1065/20 from the Israel Science Foundation, by the United States – Israel Binational Science Foundation (NSF-BSF grant no. 2022792) and by ERC grant no. 882971 "GeoScape" and by the Erd{\H o}s Center. \texttt{shakhar@bgu.ac.il}. 
        https://orcid.org/0000-0003-3038-6955}
}
\date{}
\newtheorem{theorem}{Theorem}[section]

\newtheorem{proposition}[theorem]{Proposition}
\newtheorem{lemma}[theorem]{Lemma}
\newtheorem{Q}[theorem]{Question}

\newtheorem{corollary}[theorem]{Corollary}

\newtheorem*{claim*}{Claim}

\theoremstyle{definition}
\newtheorem{definition}[theorem]{Definition}

\newcommand{\cE}{\mathcal{E}}
\newcommand{\F}{\mathcal{F}}

\DeclareMathOperator{\CF}{\chi_{cf}}
\newcommand{\II}{\mathsf{I}}

\newcommand{\FC}{\mathsf{FC}}
\newcommand{\Hm}[2]{\mathsf{H}^{\gamma_{#1}}_{#2}}

\DeclareMathOperator{\conv}{conv}
\def\Re{\mathbb R}
\def\C{\cal C}

\usepackage{todonotes}

\begin{document}

\maketitle

\begin{abstract}
 We study the conflict-free chromatic number of hypergraphs derived from the family of facets of $d$-dimensional cyclic polytopes with $n$ vertices. While in odd dimensions $d$ the problem is easy, for even dimensions the problem becomes very difficult and exhibits interesting connections to extremal graph theory. We provide sharp asymptotic bounds for the conflict-free chromatic number in several small even dimensions and non-trivial upper and lower bounds for general even dimensions. The main purpose of this paper is revealing a surprising relation between conflict-free colorings and the celebrated Erd\H{o}s girth conjecture, opening new avenues for future research.
\end{abstract}

\section{Introduction}
\label{sec:intro}
\subsection{Conflict-free coloring on the moment curve}
Given a hypergraph $H = (V, \cE)$, a vertex coloring is called \emph{conflict-free} (or \textit{CF-coloring}) if, for every hyperedge $e \in \cE$, there exists at least one vertex $v \in e$ whose color is unique within $e$. The minimum number of colors needed for a CF-coloring of $H$ is called the \emph{conflict-free chromatic number} (or \textit{CF-chromatic number}) and is denoted by $\chi_{\mathrm{cf}}(H)$. In particular, it is straightforward that $\chi_{\mathrm{cf}}(H) \geq \chi(H)$, where $\chi(H)$ is the proper (or usual) chromatic number of $H$, requiring every hyperedge of size at least two to contain at least two differently colored vertices.

Conflict-free colorings have been extensively studied in both combinatorics and computational geometry. Their initial motivation came from frequency assignment problems in cellular networks \cite{ELRS,SmPHD}. Since then, the concept has been generalized to various types of hypergraphs, both abstract and geometrically defined, leading to a vibrant research area with numerous follow-up publications. The interest in conflict-free coloring has also spread to other disciplines, including logic, graph theory, theoretical computer science, and algorithms (see, e.g., \cite{AADFGHKS17,AS08,CheilarisGRS14,cf7,cf6,FeketeK18,Hajnal2011,HS02,JartouxKSY24,KellerRS21,KellerS20a,CFPT09}). For a comprehensive overview, see the survey \cite{CF-survey} and references therein.

\smallskip

 A substantial amount of research in discrete geometry is devoted to providing sharp bounds on the CF-chromatic numbers of hypergraphs induced by geometric shapes: For a family $\C$ of shapes in $\Re^d$ we say that a hypergraph $H=(V,\cE)$ \textit{is induced by} $\C$ if $V$ is a finite set of points in $\Re^d$ and $\cE =\{ V \cap C | C \in \C\}$. The most studied cases are hypergraphs on points in $\Re^d$ induced by ``well-behaved" 
shapes such as balls, half-spaces, axis parallel-boxes or more generally, semi-algebraic sets defined by ``low-degree" polynomials.

In this paper, we are interested in hypergraphs induced by points in $\Re^d$ with respect to half-spaces. For such a hypergraph $H$ on $n$ points in $\Re^2$ or $\Re^3$, it is well known that $\CF(H)=O(\log n)$ and this bound is asymptotically sharp \cite{ELRS,HS02,smoro}. However, the situation becomes completely different in dimensions at least $4$; in fact, it can easily seen that the chromatic number can be at least $n$, the number of points, which is also a trivial upper bound. 
Such lower bounds are given e.g., by points on the  \textit{moment curve} $\gamma_d: \Re \to \Re^d$, defined by $\gamma_d(t)=(t,t^2, \dots, t^d)$. Before we discuss this in detail, let us first briefly cover some known facts about the curve, which are useful for further discussion. The moment curve has the following combinatorial properties. Every set of at most $d+1$ distinct points on $\gamma_d$ are affinely independent and any point set on $\gamma_d$ is in convex position. Many properties of points on the curve depend only on the order of the points along $\gamma_d$. So we will frequently abuse our notations and identify an $n$-point set on the moment curve with $[n]=\{1,2,\dots, n\}$. We also refer to the convex hull of a set of $n$ points on $\gamma_d$ as ``the" \textit{cyclic polytope} $\mathsf{C}_d(n)$, and identify a face of $\mathsf{C}_d(n)$ with its vertex set. Due to this combinatorial description, the moment curve has been of interest in combinatorics and discrete geometry.

 It is well-known that $\mathsf{C}_d(n)$  is a  \textit{neighborly polytope}; namely, any vertex subset $S\subseteq [n]$ with $|S|\leq d/2$ forms a face of $\mathsf{C}_d(n)$. By the celebrated upper bound theorem of McMullen \cite{upper_bound_mcmullen}, the cyclic polytope $\mathsf{C}_d(n)$ as well as all $d$-dimensional (simplicial) neighborly polytopes on $n$ vertices have the largest number of faces among all $d$-polytopes on $n$ vertices. For more backgrounds on the moment curve and neighborly polytopes, refer to \cite{lectures_on_polytopes_book}.

Taking our attention back to CF-colorings, let $\Hm{d}{n}$ be the hypergraph on $n$ points on $\gamma_d$ induced by half-spaces. Note that for $d \geq 4$, $\mathsf{C}_{d}(n)$ contains the complete graph $K_n$, namely, all pairs of vertices form a hyperedge of cardianlity $2$ since $\mathsf{C}_{d}(n)$ is neighborly. This (trivially) implies that $\CF(\Hm{d}{n})=n$ for $d \geq 4$ in contrast to dimensions $d \in \{1,2,3\}$.

Note that a similar situation happens even for a sub-hypergraph $H'$ of $\Hm{d}{n}$ which contains all hyperedges of $\Hm{d}{n}$ of a certain size $k\leq d/2$. Then $H'$ contains the complete $k$-uniform hypergraph on $n$ vertices.  Indeed, by the pigeon-hole principle it is an easy exercise to show that the CF-chromatic number of the complete $k$-uniform hypergraph is $\Omega(n/k)$.
It is therefore, natural to ask the following general question where we forbid hyperedges of small size in a sub-hypergraph:
\begin{Q} \label{question_main}
For $d \geq 4$, what is 
the CF-chromatic number of a sub-hypergraph of $\Hm{d}{n}$ without any hyperedges of size at most $d/2$?
 \end{Q}

The main difficulty of Question \ref{question_main} arises from the fact that the underlying hypergraph family is not hereditary. 
A family $\cal H$ of hypergraphs is {\em hereditary} if they are closed under projection; or formally if $H=(V,\cE)$ being in the family  $\cal H$ implies that for every $V' \subseteq V$, $H[V']$ also belongs to the family $\cal H$, where $H[V']=(V',\{e \cap V' | e \in \cE\})$.
Many hypergraph families that were studied in the literature with respect to CF-colorings are hereditary.  For example, in \cite{smoro} it was shown that if $H=(V,\cE)$ with $|V|=n$ satisfies that for every $V' \subset V$, $\chi(H[V']) \leq k$ then $\CF(H)=O(k \log n)$. Without the hereditary condition one can construct a hypergraph $H$ on $n$ vertices with $\chi(H)=2$ while $\CF(H) = \Omega(n)$, 
for example,
\[H'=([2m], \{e \subset [2m]: |e|=4,\, e\cap [m]\ne \emptyset, e\cap \{m+1,m+2, \dots, 2m\} \ne \emptyset \}) \]
has $\CF(H')=m+1$ while clearly $\chi(H')=2$ for $m\geq 3$. 

None of the hypergraphs studied in this paper exhibits such a hereditary property and therefore we cannot use any of the standard tools developed in the literature for hereditary hypergraphs. Hence we need a different approach. 

\smallskip 

In this paper, we give a partial answer to Question \ref{question_main} for the \textit{facet hypergraph} of a cyclic polytope $\FC_d(n)$ and the \textit{2-interval hypergraph} $\II_n^2$, see precise definitions below.
Our techniques combine combinatorial arguments with tools from extremal graph theory and combinatorial properties of the moment curve. In particular for cyclic polytopes, 
we exhibit a strong relation between $\CF(\FC_d(n))$ and the following famous Tur{\'a}n type problem in extremal graph theory known as the {\em girth conjecture} \cite{erdos_girth_conjecture}: \textit{What is the maximum possible number of edges that a graph with girth at least $2k+1$ can have?} For general background on such extremal problems, we refer the reader to the survey \cite{furedi_simonovits_degenerate_extremal_survey}.

This relation might be of independent interest and we believe that it can shed light on CF-coloring problems for other classes of geometric hypergraphs. For example, the so-called family of \textit{$(AB)^{l/2}$-free hypergraphs} was recently introduced in \cite{ABA} as an abstract approach for the study of coloring problems on geometric hypergraphs. See also \cite{ABAB_stabbed_pseudodisk,  dual_ABAB2024, complexity_ABAB}. As we observe in Section \ref{sec:final_remarks}, the sub-hypergraphs of $\Hm{d}{n}$ are $(AB)^{(d+2)/2}$-free hypergraphs. Hence, it is natural to ask if our approach and results can be extended to such families as well.

\subsection{Results}
\textbf{Cyclic polytopes, the Facet hypergraph and the Girth conjecture.} 
We define the \textit{facet hypergraph} $\FC_d(n)$ of the cyclic polytope $\mathsf{C}_d(n)$ by
\[\FC_d(n)=([n],\{f\subseteq [n]: \textrm{$f$ is a facet of $\mathsf{C}_d(n)$}\}).\]
Since $\mathsf{C}_d(n)$ is simplicial, $\FC_d(n)$ is $d$-uniform.

We provide asymptotically sharp bounds on $\CF(\FC_d(n))$ in several small dimensions, along with non-trivial upper and lower bounds for general even dimensions. Here we only consider even dimensions $d$ since it is easy to see that $\CF(\FC_d(n)) \leq 3$  for odd $d$, see Proposition~\ref{prop:CF_odd_cyclic} in below. 

\begin{theorem}[Lower bounds on $\CF(\FC_d(n))$ for even $d$]{} \label{thm:lower_bound_CF}
	\leavevmode
	\begin{enumerate}[(1)]
		\item $\CF(\FC_d(n)) = \Omega(\sqrt{n})$ for every even $d \geq 4$.
		\item For every positive integer $k$ and for $d=4k$, we have $\CF(\FC_d(n)) = \Omega(n^{\frac{d}{d+4}})=\Omega(n^{\frac{k}{k+1}})$.
            \item For every even dimension $d \geq 14$ with $d/2$ being odd, we have $\CF(\FC_d(n)) = \Omega(n^{\frac{k^*}{k^*+1}})$ where $k^*=\lfloor d/6 \rfloor$.
	\end{enumerate}
\end{theorem}

\begin{theorem}[Sharp bounds on $\CF(\FC_d(n))$ for small even $d$]\label{thm:tight_upper}
For $d=4,6,8,10,12$ and $20$, we have:
\begin{align*}
&\CF(\FC_4(n))=\Theta(\sqrt{n}), &&\CF(\FC_6(n))=\Theta(\sqrt{n}), &&\CF(\FC_8(n))= \Theta(n^{2/3}),\\
 &\CF(\FC_{10}(n))=\Theta(\sqrt{n}), &&\CF(\FC_{12}(n))=\Theta(n^{3/4}), && \CF(\FC_{14}(n))=\Theta(n^{2/3}), \\
 &\CF(\FC_{18}(n))=\Theta(n^{3/4})\textrm{, and} &&\CF(\FC_{20}(n))=\Theta(n^{5/6}). &&
\end{align*}
\end{theorem}

\begin{theorem}[Upper bounds for general even $d$]\label{thm:general_upper}
For even $d\geq 16$, let $k=\lfloor d/4 \rfloor$.
We have $\CF(\FC_d(n))=O(n^{1-\frac{2}{3k-1+\epsilon}})$ where $\epsilon=0$ when $k$ is odd and $\epsilon=1$ when $k$ is even. 
\end{theorem}

In fact, this particular case has a surprising connection to the Erd\H{o}s girth conjecture in extremal graph theory \cite{erdos_girth_conjecture}. The girth conjecture posits the maximum possible number of edges in a graph with given girth, a classical open problem in extremal combinatorics, for general backgrounds for the conjecture and related Tur{\'a}n type problems, see \cite{furedi_simonovits_degenerate_extremal_survey}. This connection indicates that $\CF(\FC_d(n))$ are highly non-trivial to determine for general dimensions. The connection also implies that improving the bounds on $\CF(\FC_d(n))$ would imply new bounds on the girth conjecture, making it of independent interest.

\smallskip 

\textbf{2-interval hypergraphs.} 
The \textit{2-interval hypergraph} on vertex set $[n]$, denoted by $\II_n^2$, is a hypergraph where the hyperedges are union of two (discrete) intervals with a total cardinality at least $3$. It is well-known that $\gamma_d$ crosses a hyperplane of $\mathbb{R}^d$ in at most $d$ distinct points. Hence, $\II_n^2$ can be regarded as a sub-hypergraph of $\Hm{4}{n}$ for Question \ref{question_main}. In fact, one can easily see that $\II_n^2$ contains (asymptotically) almost all the hyperedges of $\Hm{4}{n}$ of size at least 3.

Note also that $\II^2_n$ is a ``larger" hypergraph than 
$\FC_4(n)$; in size the former has $\Theta(n^4)$ hyperedges 
while the latter has $\Theta(n^2)$ hyperedges. Also $\II^2_n$ has a loose restriction for hyperedges than $\FC_4(n)$. Even so, the following Theorem shows that $\CF(\II^2_n)$ and $\CF(\FC_4(n))$ are asymptotically the same.

\begin{theorem} \label{thm:2-interval}
	$\CF(\II^2_n)=\Theta(\sqrt{n})$.
\end{theorem}

From Theorem \ref{thm:2-interval} it follows that when $d=4$ every sub-hypergraph in Question \ref{question_main} has the CF-chromatic number $O(\sqrt{n})$. 

\subsection{Organization of the paper.} In Section \ref{section:cyclic}, after observing some basic facts on colorings of $\FC_d(n)$, we prove Theorems \ref{thm:lower_bound_CF}, \ref{thm:tight_upper} and \ref{thm:general_upper}. In the same section, we briefly discuss CF-colorings of a natural generalization of $\FC_d(n)$ which concerns the disjoint union of paths on $r$ vertices on a cycle. In Section \ref{sec:2-interval} we discuss CF-colorings of the 2-interval hypergraph $\mathsf{I}_n^2$ and prove Theorem \ref{thm:2-interval}. In Section \ref{sec:final_remarks} we state several open problems.

\section{CF-chromatic numbers of cyclic polytopes} \label{section:cyclic}

\subsection{Preliminaries}

The following is a well-known combinatorial criterion which characterizes the facets of $\mathsf{C}_d(n)$.

\begin{theorem}[Gale's evenness criterion] \label{thm:gale}
A $d$-subset $S$ of $[n]$ forms a facet of $\mathsf{C}_d(n)$ if and only if  the set $\{k\in S:i<k<j \}$ has even size for every $i, j \in [n]\setminus S$. Namely, every maximal contiguous subset of $S$ which contains neither $1$ nor $n$ has even size.
\end{theorem}

Now we state some basic observations. Without loss of generality we assume that $n \geq d+1$.

\begin{proposition}[Proper colorings] \label{prop:proper_cyclic}
	For $n\geq 3$, we have $\chi(\FC_2(n))=3$ when $n$ is odd and $\chi(\FC_2(n))=2$ when $n$ is even. For $d\geq 3$ and $n \geq d+1$, $\chi(\FC_d(n))=2$. 
\end{proposition}
\begin{proof}
The first claim is obvious since $\FC_2(n)$ is a cycle of length $n$. When $d\geq 3$, we alternately color RED and BLUE by coloring odd elements of $[n]$ by RED and even elements of $[n]$ by BLUE. By Theorem \ref{thm:gale}, every hyperedge of $\FC_d(n)$ should contain an interval of size 2. Hence, it must contain two colors.
\end{proof}

\begin{proposition}[CF-colorings in odd dimensions] \label{prop:CF_odd_cyclic}
For $n \geq 4$, $\CF(\FC_3(n))=2$. For odd $d\geq 5$ and $n \geq d+1$, $\CF(\FC_d(n))=3$.
\end{proposition}
\begin{proof} Since $\FC_3(n)$ is 3-uniform, a proper 2-coloring in Proposition \ref{prop:proper_cyclic} also gives a CF-coloring of $\FC_3(n)$. So the first claim follows. 
	
It remains to show the case for odd $d \geq 5$. We first show the upper bound. The following is a CF-coloring of $\FC_d(n)$ with colors $\{\text{RED}, \text{BLUE}, \text{GREEN}\}$: we color 1 by RED, $n$ by BLUE, and the other vertices by GREEN. This is a CF-coloring since every hyperedge contains either 1 or $n$ by Theorem \ref{thm:gale} and these end vertices have unique colors RED and BLUE. 

Next, we show the lower bound. Let us consider an arbitrary 2-coloring of $\FC_d(n)$ by RED and BLUE. We show that this coloring is not a CF-coloring. Note that by Theorem \ref{thm:gale}, for any two disjoint intervals $I_1$ and $I_2$ of size 2 in $[n]$, there is a hyperedge of $\FC_d(n)$ which contains both $I_1$ and $I_2$ as subsets. Hence, we are done if we can find two intervals $I_1$ and $I_2$ of size 2 such that $I_1\cup I_2$ has 2 RED vertices and 2 BLUE vertices. Suppose otherwise.

Let $m$ be the maximum length of an increasing sequence in $[n]$ where a consecutive pair of elements have different colors from the 2-coloring. We cannot have $m\geq 4$, otherwise we can find a pair of disjoint intervals of size 2 we forbid. For the case when $m=3$, without loss of generality, we assume the color alternates RED-BLUE-RED when we move from 1 to $n$. We cannot have two BLUE vertices, otherwise we can again find a pair of intervals we forbid. Hence, there is only one BLUE vertex. Then, by Theorem \ref{thm:gale}, we can find a hyperedge of $\FC_d(n)$ which only consists of RED vertices. A similar argument holds when $m\leq 2$. This completes the proof.
\end{proof}
Hence in what follows we only consider even dimensions $d$.

\subsection{Palette graphs and lower bounds on  \texorpdfstring{$\CF(\FC_d(n))$}{Xcf(FCd(n))} for even \texorpdfstring{$d \geq 4$}{d>=4}}\label{sec:lower-bound}
In this subsection we prove Theorem \ref{thm:lower_bound_CF}. First, we prove Part (1).

\begin{proof}[Proof of Theorem \ref{thm:lower_bound_CF} (1)]
It is enough to prove it for the case when $n$ is even. For fixed $d=2l \geq 4$ and even $n\geq d+1$, put $H=\FC_d(n)$.
Suppose we are given a CF-coloring $\varphi: [n] \to [c]$ of $H$.
By Theorem \ref{thm:gale}, every $d$-subset of $[n]$ of the form $\{2i_1-1,2i_1,2i_2-1,2i_2, \dots, 2i_l-1,2i_l\}$ for some  $1 \leq i_1 < i_2 < \cdots <i_l \leq \frac{n}{2}$, is a hyperedge of $H$.
Note that we cannot have all the $l$ sets $\{\varphi(2i_1-1),\varphi(2i_1)\}$, $\{\varphi(2i_2-1),\varphi(2i_2)\}$, $\dots, \{\varphi(2i_l-1),\varphi(2i_l)\}$ be pairwise equal.
For otherwise the corresponding hyperedge has each color appearing at least $l$ times.
Since there are $\frac{n}{2}$ such sets $\{2i-1,2i\}$ and at most $\binom{c}{2}+c$ unordered pairs of colors (not necessarily distinct) we have that
\[ \frac{n}{2} \leq \frac{c^2+c}{2}\cdot (l-1),\]
	so $c = \Omega(\sqrt{n})$. 
\end{proof}

Next, we introduce the following notion of a palette graph that will be useful for us:
\begin{definition} \label{def:palette}
Given a (simple) graph $G=(V,E)$ and a coloring $\varphi: V \to [c]$, the {\em palette graph} of $G$ with respect to $\varphi$, denoted by $P_{G,\varphi}$, is the multigraph $H=([c],\F)$ such that
 there is a bijection $\psi: E \to \F$ where, for every edge $e=\{v_1, v_2\}$, $\psi(e)$ is an edge between $\varphi(v_1)$ and $\varphi(v_2)$.
\end{definition}

For a positive integer $n$, let \[M_{[n]}=([n], \{\{2i-1,2i\}: \textrm{$i$ is an integer with }1 \leq i \leq n/2\}).\] The main ingredient in the proof of Theorem \ref{thm:lower_bound_CF} (1) can be restated using the palette graph $P_{M_{[n]}, \varphi}$. Namely, for any two colors $i$ and $j$ of $\varphi$, not necessarily distinct, a multi-edge $\{i,j\}$ has multiplicity at most $d/2-1$ in $P_{M_{[n]}, \varphi}$. 
	
In proving Parts (2) and (3) of Theorem \ref{thm:lower_bound_CF}, we also show that a certain restriction on a palette graph is unavoidable. Note that in a multi-graph $H$ and a vertex $v$ of $H$, a loop incident to $v$ is counted twice in the degree of $v$  and other edges incident to $v$ are counted exactly once. 

\begin{lemma}\label{lemma_lower_no_subgraph}
Let $d=2l$ and $n\geq d+1$. Let $\varphi:[n]\to [c]$ be a CF-coloring of $\FC_d(n)$ using $c$ colors. The palette graph $P=P_{M_{[n]},\varphi}$ contains no submultigraphs with exactly $l$ edges and minimum degree at least $2$.
\end{lemma}
\begin{proof}
Assume to the contrary that there is a submultigraph $H$ in $P$ with the specified conditions. The $l=d/2$ edges of $H$ correspond to distinct edges $e_1, \dots, e_l$ of $M_{[n]}$ by definition. Note that $e=e_1\cup e_2 \cup \cdots \cup e_l$ is a hyperedge of $\FC_d(n)$, and the degree condition on $H$ implies that $e$ does not attain a unique color by $\varphi$, a contradiction. This completes the proof.
\end{proof}

We also need the following lemma which seems to be of independent interest in extremal graph theory:
\begin{lemma} \label{lemma:upper_on_palette}
Let $G$ be a simple graph on $c$ vertices which contains no subgraphs with exactly $l$ edges and minimum degree at least $2$. Then,
     \[
    |\cE(G)| = \left\{\begin{array}{ll}
        O(c^{1+(1/k)}) & \text{when $l$ is even,}\\
        O(c^{1+(1/k^*)}) & \text{when $l$ is odd and $l\geq 7$,}
        \end{array}\right.
  \]
where $k=l/2$ and $k^*=\lfloor l/3 \rfloor$.
\end{lemma}
To prove Lemma~\ref{lemma:upper_on_palette} we apply several known results from extremal graph theory. Before proceeding to the proof of Lemma~\ref{lemma:upper_on_palette}
we first provide the proof of Theorem \ref{thm:lower_bound_CF} (2) and (3).

\begin{proof}[Proof of Theorem \ref{thm:lower_bound_CF} (2) and (3)]
We only show it for the case when $n$ is even; the other case will follow similarly. Given an even dimension $d=2l \geq 4$ and $n \geq d+1$, put $H= \FC_{d}(n)$. Suppose that $\varphi: [n]\to [c]$ is a CF-coloring of $H$. Put $P=P_{M_{[n]}, \varphi}$.

Let $G$ be the simple graph obtained from $P$ by removing all loops and leaving exactly one copy of a multi-edge 
for each pair of distinct vertices of $P$ if such pair has an edge in $P$. Note that in $P$, by Lemma \ref{lemma_lower_no_subgraph}, the multiplicity of any  edge (or loop) is at most $l-1=d/2-1$. Therefore, we obtain the following inequality:
\[n/2=|\cE(P)| \leq (l-1)|\cE(G)|+(l-1)c.\] 

Note also that $G$ satisfies the condition of Lemma \ref{lemma:upper_on_palette} by applying Lemma \ref{lemma_lower_no_subgraph} on $P$. So, by Lemma \ref{lemma:upper_on_palette}, we get $n=O(c^{1+(1/\tilde{k})})$ or alternatively
$$
c = \Omega(n^{\frac{\tilde{k}}{\tilde{k}+1}})
$$ where $\tilde{k}=d/4$ when $l$ is even, and $\tilde{k}=\lfloor d/6 \rfloor$ when $l$ is odd and $l\geq 7$. This completes the proof.
\end{proof}

Next, we prove Lemma \ref{lemma:upper_on_palette}. The proof relies on the following two theorems. The first theorem attributed to Erd\H{o}s, whose first proof was published by Bondy and Simonovits \cite{even_cycle_thm_bondy_simonovits}. Denote by $ex(n,H)$ the maximum number of edges that a graph $G$ on $n$ vertices can have under the condition that $G$ does not contain $H$ as a subgraph of $G$.

\begin{theorem}[The Even Cycle Theorem]	\label{thm:even_cycle_forbid_upper_bound}
	$ ex(n,C_{2k})=O(n^{1+(1/k)}).$
\end{theorem}	

Define the \textit{generalized theta graph}, denoted by $\Theta_{k_1, \dots, k_m}$, to be the graph obtained by fixing two vertices $v$ and $v'$, which are connected by $m$ internally disjoint paths with lengths $k_1, \dots, k_m$, respectively. The following result is from \cite{generalized_theta_liu_yang}.

\begin{theorem} \label{thm:generalized_theta}
Fix positive integers $k_1, \dots, k_m$ with the same parity, in which 1 appears at most once. Then,
\[ex(n,\Theta_{k_1, \dots, k_m})=O(n^{1+(1/k^*)}),\]
where $k^*=\frac{1}{2}\min_{1\leq i < j \leq m}(k_i+k_j)$.
\end{theorem}

The following is an easy implication from Theorem \ref{thm:generalized_theta}.

\begin{corollary}\label{cor:generalized_theta}
Let $l \geq 7$ be odd. Let $k_1, k_2$ and $k_3$ be odd positive integers which almost equipartition $l$, that is, satisfy $|k_i-k_j|\leq 2$ for every $i,j \in [3]$ and $k_1+k_2+k_3=l$. Then, 
\[ex(n,\Theta_{k_1, k_2, k_3})=O(n^{1+(1/k^*)}),\]
where $k^*=\lfloor l/3 \rfloor$.
\end{corollary}
\begin{proof}[Proof of Lemma \ref{lemma:upper_on_palette}]
When $l$ is even, $G$ cannot have a cycle $C_l$ as a subgraph by the condition. When $l$ is odd and at least 7, for positive odd integers $k_1, k_2$ and $k_3$ which almost equipartition $l$, $G$ cannot have $\Theta_{k_1,k_2,k_3}$ as a subgraph by the condition. Therefore, Theorems \ref{thm:even_cycle_forbid_upper_bound} and \ref{thm:generalized_theta} give the desired result.
\end{proof}

\subsection{Upper bounds on \texorpdfstring{$\CF(\FC_d(n))$}{Xcf(FCd(n))} for even \texorpdfstring{$d\geq 4$}{d>=4}}\label{sec:upper-bound}
\begin{figure}[ht]
	\centering
	\includegraphics[totalheight=6cm]{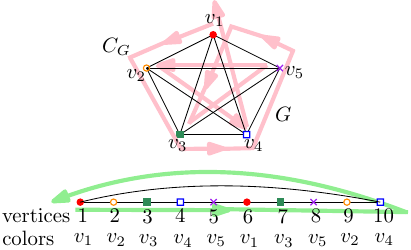}
	\caption{Illustration: A coloring on $[n]$ using $c$ colors from an Eulerian circuit $C_G$ of a graph $G$ with $c$ vertices and $n$ edges. For this example, we have $n=10$ and $c=5$, and the tour begins from $v_1$ and $1$ for $G$ and $C_{[10]}$, respectively. By the proof of Theorem \ref{thm:tight_upper} below, this also gives a CF-coloring of $\FC_4(10)$.
 }
	\label{fig:C4n}
\end{figure}

Before we turn our attention to prove upper bounds on $\CF(\FC_d(n))$ we state the following key lemma:

\begin{lemma} \label{lemma:sufficient_condition_palette}
	Fix an even dimension $d\geq 4$. Let $G$ be a simple graph on $c$ vertices and $n$ edges with the following properties.
	\begin{enumerate}[(1)]
		\item $G$ is Eulerian, that is, every vertex of $G$ has an even degree and $G$ is connected.
		\item $G$ does not contain any cycle $C_j$ where $3 \leq j \leq d/2$ and $j \ne d/2-1$. 
	\end{enumerate}
	Then, there is a CF-coloring of $\FC_{d}(n)$ with $c$ colors.
\end{lemma}
We use Lemma \ref{lemma:sufficient_condition_palette} and lower bounds construction of graphs with large girth and many edges in order to provide upper bounds on $\CF(\FC_d(n))$.

\smallskip 

We first prove Lemma \ref{lemma:sufficient_condition_palette}. For $n\geq 3$, recall that we denote by $C_{[n]}$ the cycle on the vertex set $[n]$ where the elements of $[n]$ are cyclically ordered. That is, let
\[C_{[n]}=([n], \{\{i,i+1\}: i\in [n]\textrm{ in modulo $n$}\}).\]

It is easy to see that given a coloring $\varphi:[n]\to [c]$ using $c$ colors, the palette graph $P_{C_{[n]}, \varphi}$ is Eulerian. For the other direction, given an Eulerian graph $G$ on $c$ vertices and $n$ edges, we can also find a coloring $\varphi:[n] \to [c]$ such that $G$ is the same as the palette graph $P_{C_{[n]}, \varphi}$: choose an Eulerian circuit $C_G$ of $G$ and we simultaneously traverse $C_{[n]}$ and $C_G$. Whenever we visit a new vertex $v$ of $C_{[n]}$, we color $v$ with the current vertex of the graph $G$. We do it until we traverse all the edges and return to the original vertex. See Figure \ref{fig:C4n} for an illustration.

Note that for even dimensions $d$ we have
\[\cE(\FC_d(n))=\{e_1\cup \cdots \cup e_{d/2} \subset [n]: e_i \in \cE(C_{[n]}),\, e_i\cap e_j=\emptyset \textrm{ for distinct $i,j \in [n]$}   \}\]
by Theorem \ref{thm:gale}. Lemma \ref{lemma:sufficient_condition_palette} can be also understood as a theorem about palette graphs of $C_{[n]}$ before colorings become explicit, but the conditions in the theorem also provide a concrete CF-coloring of $\FC_d(n)$ for even $d\geq 4$.

\begin{proof}[Proof of Lemma \ref{lemma:sufficient_condition_palette}]
	By (1), $G$ is Eulerian. Take an Eulerian tour $e_1e_2\dots e_n$ of $G$ and construct a coloring $\varphi: [n]\rightarrow [c]$ as described above. From $\varphi$, we can find an induced bijection $\psi:\cE(C_{[n]})\to \cE(G)$. We claim that $\varphi$ is a CF-coloring of $\FC_d(n)$, that is, we show that an arbitrary hyperedge $e$ of $\FC_d(n)$ has a vertex with a unique color. Assume to the contrary that this is not the case and let $l=d/2$. By Theorem \ref{thm:gale}, there are pairwise disjoint edges $e_1, \dots, e_l$ of $C_{[n]}$ such that $e=e_1 \cup \cdots \cup e_l$. Since $G$ is simple, by our construction of $\varphi$, each 
	$e_i$ attains exactly 2 colors from $\varphi$. So, every color used in $\varphi(e)$ should appear in at least two edges among $e_1, \dots, e_l$. Let $H$ be the subgraph of $G$ induced by the edges $\psi(e_1), \dots, \psi(e_l)$. It follows that every vertex of $H$ has degree at least 2, so there is a cycle $C_H$ in $H$ of length at most $l$. Since the other lengths are forbidden, $C_H$ has length $l-1$ (this implies that $l\geq 4$). Then there is exactly one edge of $H$, say $f=\psi(e_1)$, which is not used in $C_H$. If $f$ is a chord of $C_H$, then we can find a smaller cycle which leads to a contradiction with (2). If $f$ uses a vertex not in $C_H$, then $H$ has a vertex of degree 1 which again leads to a contradiction.
\end{proof}

\smallskip

\textbf{Explicit Upper Bounds on $\CF(\FC_d(n))$.}  Now we prove Theorems \ref{thm:tight_upper} and \ref{thm:general_upper} by using Lemma \ref{lemma:sufficient_condition_palette} and known constructions on the girth conjecture, see \cite{furedi_simonovits_degenerate_extremal_survey} for backgrounds. 
For each case, we first find a graph $G$ on $c$ vertices and $n$ edges which satisfies Conditions (1) and (2) from Lemma \ref{lemma:sufficient_condition_palette} for a given dimension $d$ where $c$ and $n$ can be arbitrarily large. Then we express $c$ in terms of $n$, and use Lemma \ref{lemma:sufficient_condition_palette} to conclude $\CF(\FC_d(n))\leq c(n)$. 

For many values of $n$, the presented constructions might not have exactly $n$ edges where we cannot directly apply Lemma \ref{lemma:sufficient_condition_palette}.
However, in each construction of dimension $d$, one can easily find a suitable constant $\alpha > 1$ such that for a sufficiently large $n$ there is a desirable graph using $n'$ edges with $n\leq n' \leq \alpha n$, which in turn gives a CF-coloring $\varphi$ of $\FC_d(n')$ by Lemma \ref{lemma:sufficient_condition_palette}. Note that the restriction of $\varphi$ to $\FC_d(n)$ is a CF-coloring which gives the same asymptotic bound on the number of colors up to a constant factor.

\begin{proof}[Proof of Theorem \ref{thm:tight_upper}] Note that we already have the desired lower bound for each case.
We separately consider the upper bound for each dimension.

\begin{claim*}
$\CF(\FC_4(n))=O(\sqrt{n}).$
\end{claim*}
\begin{proof}
	(See Figure \ref{fig:C4n} for an illustration.) For this case, we do not need to forbid any cycles. So, we can take $G=K_{c}$ where $c$ is odd to satisfy (1). Since $G$ has	$n=\binom{c}{2}$ edges, we have $\CF(\FC_4(n))\leq c = O(\sqrt{n})$.
\end{proof}

\begin{claim*}
	$ \CF(\FC_6(n)) = O(\sqrt{n}).$
\end{claim*}
\begin{proof}
	For this case, we need to forbid triangles $C_3$. For $c=4m$, we take $G=K_{2m,2m}$. $G$ satisfies (1) and (2) for $d=6$. Since $G$ has $n=c^2/4$ edges, we have $\CF(\FC_6(n))\leq c = \sqrt{n}$.
\end{proof}

\begin{claim*}
	$\CF(\FC_8(n)) = O(n^{\frac{2}{3}})$
\end{claim*}
\begin{proof}
For this case, we need to forbid $C_4$. We use finite projective planes. A \textit{finite projective plane of order $q$} consists of a set $X$ of $q^2+q+1$ elements called \textit{points}, and a family $\mathcal{L}$ of $q^2+q+1$ subsets of $X$ called \textit{lines}, which satisfies the following properties:
\begin{enumerate}[(i)]
	\item Each line has $q+1$ points.
	\item Any point belongs to exactly $q+1$ lines.
	\item Every two points lie on a unique line.
	\item Any two lines meet in a unique point.
\end{enumerate}
When $q$ is a prime power, there is a well-known construction of a finite projective plane of order $q$, $PG(2,q)$, from a finite field $\mathbb{F}_q$. For more details on finite projective planes, refer to \cite[Section 12.4]{jukna_book}.

Next, for an odd prime power $q$, let $X$ and $\mathcal{L}$ be the set of points and lines of $PG(2,q)$, respectively. We construct a bipartite graph $G$ on the vertex set $X\cup \mathcal{L}$ such that $p\in X$ and $l\in \mathcal{L}$ are adjacent in $G$ if and only if $p$ is on $l$. By (i) and (ii), $G$ is $(q+1)$-regular, so all degrees in $G$ are even. It is also easy to see $G$ is connected by (iii) and (iv). Hence $G$ is Eulerian. Also, $G$ does not contain $C_4$ as a subgraph by (iii) (or (iv)). 

Denote the number of vertices and edges of $G$ by $c$ and $n$, respectively. Then, $c=2q^2+2q+2$ and $n=(q+1)(q^2+q+1)$. This implies that $\CF(\FC_8(n))\leq c =O(n^{\frac{2}{3}})$.
\end{proof}

\begin{claim*}
	$\CF(\FC_{10}(n)) = O(\sqrt{n}).$
\end{claim*}

\begin{proof}
	For this case, we need to forbid $C_3$ and $C_5$. So it is enough to take $G=K_{\frac{c}{2}, \frac{c}{2}}$, and we can use the same argument for $\FC_6(n)$.
\end{proof}

\begin{claim*}
	$\CF(\FC_{12}(n)) = O(n^{\frac{3}{4}}).$
\end{claim*}
\begin{proof}
For this case, we need to forbid $C_6$, $C_4$ and $C_3$. In \cite{benson_c6_c10}, for every prime power $q$, Benson constructed a minimal bipartite $(q+1)$-regular graph $B_q$	of girth 8. So $B_q$ does not contain any of the cycles we forbid. Also, for an odd prime power $q$, all degrees in $B_q$ are even. Furthermore, $B_q$ is connected, otherwise it is not a minimal $(q+1)$-regular graph $B_q$	of girth 8. Hence, $B_q$ satisfies all conditions (1) and (2).

From the construction, $B_q$ has $c=2(q^3+q^2+q+1)$ vertices. So $B_q$ has $n=(q+1)(q^3+q^2+q+1)$ edges. This implies that $\CF(\FC_{12}(n))\leq c =O(n^{\frac{3}{4}})$.
\end{proof}
\begin{claim*}
	$\CF(\FC_{14}(n)) = O(n^{\frac{2}{3}}).$
\end{claim*}
\begin{proof}
For this case, we need to forbid $C_7, C_5, C_4$, and $C_3$. We can use the same construction for $\FC_8(n)$. 
\end{proof}

\begin{claim*}
	$\CF(\FC_{18}(n)) = O(n^{\frac{3}{4}}).$
\end{claim*}
\begin{proof}
 For this case, we need to forbid $C_9, C_7, C_6, \dots, C_3$.
 We can use the same construction for $\FC_{12}(n)$. 
\end{proof}

\begin{claim*}
	$\CF(\FC_{20}(n)) = O(n^{\frac{5}{6}}).$
\end{claim*}

\begin{proof}
For this case, we need to forbid $C_{10}, C_8, C_7, \dots, C_3$. In \cite{benson_c6_c10}, Benson also constructed a minimal bipartite $(q+1)$-regular graph $B_q^*$ of girth 12. Similarly as in above, by taking odd $q$, $B_q^*$ satisfies (1) and (2). $B_q^*$ has $c=2(q^5+q^4+q^3+q^2+q+1)$ vertices, so has $(q+1)(q^5+q^4+q^3+q^2+q+1)$ edges. This implies $\CF(\FC_{12}(n))\leq c =O(n^{\frac{5}{6}})$.\end{proof}
This completes the proof.
\end{proof}
\begin{proof}[Proof of Theorem \ref{thm:general_upper}]
In \cite{general_high_girth_CD_Lazebnik_Ustimenko_Woldar}, Lazebnik, Ustimenko and Woldar gave constructions $CD(a,q)$ for every integer $a \geq 1$ and prime power $q$ which yields the best known lower bound for the girth conjecture. We are particularly interested in the case when $a=2k-3$ and $q=2^t$ with $k\geq 4$ and 
$t>2$. In this case, combining the results from  \cite{general_high_girth_CD_Lazebnik_Ustimenko_Woldar} and \cite{number_component_turan_cycle}, $CD(a,q)$ satisfies the following properties.
\begin{enumerate}[(i)]
\item $CD(a,q)$ is connected and $q$-regular.
\item $CD(a,q)$ has girth at least $2k+2$.
\item $|V(CD(a,q))|=2q^{\frac{3k-3+\epsilon(k)}{2}}$.

\item  
 $|\cE(CD(a,q))|=q^{1+\frac{3k-3+\epsilon(k)}{2}}=\left(\frac{|V(CD(k,q))|}{2}\right)^{1+\frac{2}{3k-3+\epsilon(k)}}$ 
\end{enumerate}
Here $\epsilon(k)=0$ when $k$ is odd and $\epsilon(k)=1$ when $k$ is even.

Now, for even $d\geq 16$, let $k=\lfloor \frac{d}{4}\rfloor$. Then $k\geq 4$ and we have either $d=4k+2$ or $d=4k$. By (i) and (ii), we can easily see that for every positive integer $t$, $CD(a,q)$ is Eulerian and does not contain any cycles of length at most $d/2$. Denote the number of vertices and edges of $CD(a,q)$ by $c$ and $n$ respectively. By (iv), we have $\CF(\FC_d(n)) \leq c = O(n^{1-\frac{2}{3k-1+\epsilon(k)}})$. 
\end{proof}

\subsection{Unions of two disjoint intervals of size \texorpdfstring{$r$}{r} in a cycle} \label{subsection:universal_cycle}

We briefly discuss a natural generalization of $\FC_d(n)$. Let $\mathsf{D}_r^m(n)$ be the $(mr)$-uniform hypergraph on the vertex set $[n]$ with the hyperedge-set
\[\cE(\mathsf{D}_r^m(n))=\left\{\bigcup_{i=1}^mV(P_i): \textrm{$P_1,\dots, P_m$ are pairwise vertex-disjoint paths on $r$ vertices in $C_{[n]}$}  \right\}.\]
Note that $\FC_{2m}(n)=\mathsf{D}^m_2(n)$. It was also recently shown that for the sub-hypergraph $H'$ of $\mathsf{D}^m_2(n)$ whose hyperedges are obtained as the union of subpaths $P_1, \dots, P_m$ of the big path $P_{[n]}$ on $[n]$, $H'$ is the facet hypergraph of a piecewise-linear $(2m-1)$-dimensional ball, see \cite[Lemma 4.5]{novik2024transversalnumberssimplicialpolytopes}. In this sense, $\mathsf{D}^m_2(n)$ can be regarded as both combinatorial and topological generalization of $\FC_{2m}(n)$. 

\smallskip 

For $m=2$, we obtain the asymptotically sharp bound on $\CF(\mathsf{D}_r^2(n))$ for every $r\geq 2$. 

\begin{theorem} \label{thm:m_r_intervals}
$\CF(\mathsf{D}_r^2(n))=\Theta(n^{1/r})$.
\end{theorem}

An important ingredient to the proof of Theorem \ref{thm:m_r_intervals} is universal cycles for subsets. A \textit{universal cycle for} $\binom{[c]}{r}$ is a cyclic sequence with $\binom{c}{r}$ elements from $[c]$, such that every $r$ consecutive elements are distinct and every element of $\binom{[c]}{r}$ appears exactly once consecutively. These objects were first studied by Chung, Diaconis and Graham \cite{universal_cycle_Chung_Diaconis_Graham}. There, the authors provided a necessary condition for the existence of a universal cycle, namely, that $r$ divides $\binom{c-1}{r-1}$. They also conjectured that this condition is also sufficient for a sufficiently large $c$. Since then, partial results were obtained \cite{universal_cycle_exact_Jackson, universal_cycle_exact_Hurlbert,
near_universal_cycle_subsets_Lonc2} as well as approximate versions \cite{universal_approx_blackburn,near_universal_cycle_subsets_Curtis,near_universal_cycle_subsets_Lonc,near_universal_cycle_subsets_Lonc2}. The conjecture was recently proved affirmatively by Glock, Joos, K\"{u}hn and Osthus \cite{euler_tours_hypergraphs_kuhn}:

\begin{theorem}
For every positive integer $r$, there exists $c_0$ such that for all $c \geq c_0$, there exists a universal cycle for $\binom{[c]}{r}$ whenever $r$ divides $\binom{c-1}{r-1}$.
\end{theorem}

Since we use similar argument as in above, we only provide proof sketch of Theorem \ref{thm:m_r_intervals}.
\begin{proof}[Proof sketch of Theorem \ref{thm:m_r_intervals}]
For the lower bound, recall that the proof of Theorem \ref{thm:lower_bound_CF} (1) for $\FC_4(n)$ uses an upper bound on the number of colors assigned to unordered pairs of the form $\{2i-1, 2i\}$ in $[n]$ when a CF-coloring of $\FC_4(n)$ using $c$ colors is given. For $\mathsf{D}_r^2(n)$, we rather use $r$-subsets of $[n]$ of the form $\{ri-r+1,ri-r+2,\dots, ri\}$ in $[n]$ and then apply the same argument.

For an upper bound, recall that for $\FC_4(n)$ we make use of an Eulerian circuit of the complete graph $K_c$ for odd $c$ to CF-color $\FC_4(n)$. We use a similar argument for $\mathsf{D}_r^2(n)$ except for that we use a universal cycle for $\binom{[c]}{r}$. Put $n=\binom{c}{r}$. For $[n]$, by following a universal cycle for $\binom{[c]}{r}$ and $C_{[n]}$ simultaneously, we assign the current element of the universal cycle to the current vertex of $C_{[n]}$ as a color. This gives a CF-coloring of $\mathsf{D}_r^2(n)$.
\end{proof}

If we can follow the same line of proof 
for larger $m$ as for $\FC_d(n)$, the following question might suggest a Tur{\'a}n type problem for hypergraphs, which is generally known to be very challenging.

\begin{Q} \label{question_m_r_intervals}
What is $\CF(\mathsf{D}_r^m(n))$ for positive integers $m$ and $r$? 
\end{Q}

\section{CF-chromatic number of the 2-interval hypergraph} \label{sec:2-interval} 
A set of integers is called a \textit{(discrete) interval} if it consists of consecutive integers. The \textit{2-interval hypergraph} on $[n]$ is
\[\II^2_n=\{I_1\cup I_2: I_1, I_2\subseteq [n],\, |I_1\cup I_2| \geq 3, \textrm{ and $I_1$ and $I_2$ are intervals} \}.\]
In this section, we prove Theorem \ref{thm:2-interval}, that is, we show that $\CF(\II^2_n)=\Theta(\sqrt{n})$. 

\smallskip 

Recall that from Theorem \ref{thm:tight_upper}
we have $\CF(\FC_4(n))= \Theta(\sqrt{n})$. We also have $\CF(\II^2_n)+1\geq \CF(\FC_4(n))$; we assign a new color to $1$ (or $n$) from a CF-coloring of $\II^2_n$ and this gives a CF-coloring of $\FC_4(n)$. Therefore, we have the following.
\begin{proposition} \label{prop:lowe-2-interval}
	$\CF(\II^2_n)=\Omega(\sqrt{n})$.
\end{proposition}

For the upper bound, we need some preparations. We follow a similar approach as in Section \ref{section:cyclic} finding a suitable palette graph. Our coloring of $\II^2_n$ is based on the Hamiltonian path decomposition of the complete graph $K_{2k}$ for a positive integer $k$ due to Walecki \cite{walecki}.

\begin{figure}[ht]
	\centering
	\includegraphics[totalheight=6cm]{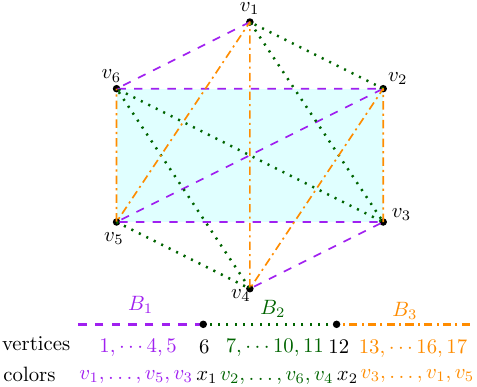}
	\caption{Walecki's Hamiltonian path decomposition and illustration of a CF-coloring of $\II^2_n$ when $k=3$. For the rectangle $\conv\{v_2,v_3, v_5, v_6\}$, the edge $v_2v_6$ and $v_3v_5$ are interior edges.}
	\label{fig:Walecki}
\end{figure}
See Figure \ref{fig:Walecki} for when $k=3$. In the figure, $K_6$ is decomposed into $3$ Hamiltonian paths of dashed (purple), dotted (green) and dash-dotted (orange) lines, respectively. For example, the dashed path is a zigzag path inside the regular hexagon starting from $v_1$ and ending with $v_4$ antipodal to $v_1$. The other paths are obtained by rotating the dashed path.

Now we formally define Walecki's decomposition of $K_{2k}$ for every $k$ in the following way. We denote the vertices of $K_{2k}$ by $v_1, v_2, \dots, v_{2k}$.  Let $D_{2k}$ be the cycle $v_1v_2\dots v_{2k}v_1$. In $D_{2k}$, there are two paths from $v_1$ to $v_{k+1}$, say $P_1$ and $P_2$, where $P_2$ contains $v_2$. We construct a Hamiltonian (zigzag) path $Z_1$ of the complete graph $K_{2k}$ on $v_1, \dots, v_{2k}$ from $v_1$ to $v_{k+1}$ such that starting from $P_1$, we take the vertices from $P_1$ and from $P_2$ alternatingly, by simultaneously following $P_1$  and $P_2$ from $v_1$ to $v_{k+1}$. For $i \in [k]$, from $Z_1$ we obtain the path $Z_i$ which contains $v_i$ as its end vertex by applying the rotation action $v_j\mapsto v_{j+i-1}$ in modulo $2k$.

\begin{proof}[Proof of Theorem \ref{thm:2-interval}]
	The lower bound is shown in Proposition \ref{prop:lowe-2-interval}. It is enough to give a CF-coloring of $\II^2_n$ which uses $O(\sqrt{n})$ colors. We fix a positive integer $k$. We show that the 2-interval hypergraph on $[n]=[2k^2-1]$ has a CF-coloring which uses $O(\sqrt{n})$ colors. 
	
	(See Figure \ref{fig:Walecki} for illustration for when $k=3$.)
 In $[n]$, by choosing suitable $k-1$ separators which form a subset $S \subset [n]$, we divide $[n]\setminus S$ into $k$ blocks 
 $B_1, \dots, B_k$ of consecutive numbers of the same size $2k-1$ where the indices are given in an increasing order. In each block $B_i$, we simultaneously follow elements of $B_i$ in an increasing order and the Hamiltonian path $Z_i$ of Walecki's construction (of the complete graph on $2k$ vertices) from $v_i$ right before arriving at $v_{i+k}$, and assign the current vertex $v_j$ of $Z_i$ to the current element $t$ of $B_i$ as the color of $t$. We also assign an additional color $x_i$ to the $i$th separator for $i\in [k-1]$. This completes the coloring. Note that we used $3k-1=O(\sqrt{n})$ colors. 
	
	Now we show that it is a CF-coloring of $\II^2_n$. Consider two intervals $I_1$ and $I_2$ in $[n]$. We show that there is a unique color in $I_1 \cup I_2$. If $I_1$ or $I_2$ contains a separator, then we are done because separators have unique colors. So both $I_1$ and $I_2$ are contained in a single block, say $B_{i_1}$ and $B_{i_2}$, respectively. The case when $B_{i_1}=B_{i_2}$ is also easy, because all vertices attain distinct colors in each block. So $B_{i_1}\ne B_{i_2}$. By the same reason, we may assume $|I_1|=|I_2|$ since otherwise there is a unique color in $I_1 \cup I_2$.
	
	Hence, the remaining case is when there are distinct blocks $B_{i_1}$ and $B_{i_2}$ with $I_1 \subseteq B_{i_1}$, $I_2 \subseteq B_{i_2}$ and $2 \leq |I_1|=|I_2|\leq 2k-1$. Denote our coloring of $\II^2_n$ by $\varphi$. It is enough to show that $\varphi(I_1) \ne \varphi(I_2)$. Suppose otherwise. 
	
	Consider the regular $(2k)$-gon $D$ in the plane where we identify the vertices of $D$ with the colors $v_1, \dots, v_{2k}$. For each $j\in [2]$, by construction $\conv(\varphi(I_j))$ has 1 or 2 boundary edges which are not edges of $D$ since $|I_j|<2k$. We call them \textit{interior edges} of $\conv(\varphi(I_j))$. Note that $Z_{i_j}$ is a unique zigzag path among
	$Z_1, \dots, Z_{k}$ which contains all interior edges of $\conv(\varphi(I_j))$: the uniqueness can be shown by that no two zigzag paths among $Z_1,\dots, Z_{k}$ share a common edge. Since $\conv(\varphi(I_1))=\conv(\varphi(I_2))$, the two convex hulls have the same set of interior edges, and the same unique zigzag paths containing them, that is, $Z_{i_1}=Z_{i_2}$. This implies $i_1=i_2$ which leads to a contradiction with that $B_{i_1}$ and $B_{i_2}$ are distinct.
\end{proof}

\section{Discussion and open problems} \label{sec:final_remarks}

\textbf{Excluding subgraphs with exactly $l$ edges and minimum degree at least $2$.} Note that we can use the following weaker condition than Condition (2) in Lemma \ref{lemma:sufficient_condition_palette}.
\begin{itemize}
    \item[(2')] $G$ does not contain a subgraph with exactly $l=d/2$ edges and minimum degree at least $2$.
\end{itemize}
As seen from this and previously in Section \ref{sec:lower-bound}, the following question has close ties to the CF-coloring problem of $\FC_d(n)$.
\begin{Q}
How many edges can a graph on $n$ vertices have if it contains no subgraphs with exactly $l$ edges and minimum degree at least $2$?  
\end{Q}
\smallskip

\textbf{CF-chromatic number of the $m$-interval hypergraph $\mathsf{I}_n^m$.}
For $m\geq 2$, the \textit{$m$-interval hypergraph} on $[n]$, which we denote by $\II^m_n$, is the hypergraph where the vertex set is $[n]$ and the hyperedge-set is
\[\cE(\II^m_n)=\left\{ \bigcup_{i=1}^mI_i: \left|\bigcup_{i=1}^mI_i\right| \geq m+1, \textrm{ and $I_1,\dots, I_m$ are intervals in $[n]$} \right\}.\]
As mentioned at introduction for 2-interval hypergraphs,  
similarly $\II^m_n$ contains (asymptotically) almost all the hyperedges of $\Hm{2m}{n}$ of size at least $m+1$.
Note also that Theorem \ref{thm:lower_bound_CF} on $\FC_{2m}(n)$ gives a lower bound for $\CF(\II^m_n)$.

\begin{Q}
What is $\CF(\II^m_n)$ when $m \geq 3$? Is it asymptotically same as $\CF(\FC_{2m}(n))$?
\end{Q}
Regarding 
Question \ref{question_m_r_intervals},
we can also ask the following question.
\begin{Q}
In the definition of $\II^m_n$, what if we give a lower bound condition on the size of each interval? How does it affect $\CF(\II^m_n)$?
\end{Q}

\textbf{$(AB)^{l/2}$-free Hypergraphs.}
For a positive integer $l$, a hypergraph $H=(V,\cE)$ on ordered vertex set $V$
is called $(AB)^{l/2}$\textit{-free} if there is no sequence of $l$ ordered vertices $v_{1},\ldots,v_{l}$ such that the odd-indexed vertices $v_1, v_3, \dots$ are in $A\setminus B$ and the even indexed vertices $v_2, v_4, \dots$ are in $B\setminus A$ for some hyperedges $A$ and $B$ of $H$. 
The notion of $(AB)^{l/2}$-free hypergraphs
has been actively studied recently in context of coloring problems on geometric hypergraphs, see, e.g., \cite{ABAB_stabbed_pseudodisk, ABA, dual_ABAB2024, complexity_ABAB}.

As we already noticed, $\Hm{d}{n}$ and $\II^m_n$ has natural orders on their vertex sets. Following those orders, note that $\II^m_n$ is an $(AB)^{(2m+1)/2}$-free hypergraph: Suppose that there are hyperedges $A$ and $B$ and ordered vertices $v_1, v_2, \dots, v_{2m+1}$ such that $v_i \in A\setminus B$ when $i$ is odd and $v_{i} \in B\setminus A$ when $i$ is even. Then, $A$ cannot consist of only $m$ intervals. One can easily see that the number $(2m+1)/2$ is tight for $\II^m_n$, that is, $\II^m_n$ is not $(AB)^{l/2}$-free for $l\leq 2m$.

Further $\Hm{d}{n}$ is an $(AB)^{(d+2)/2}$-free hypergraph and again the number $(d+2)/2$ is tight (note that there is a difference between $\II^m_n$ and $\Hm{2m}{n}$ regarding the $(AB)^{l/2}$-free property): 
Let us denote the vertex set of $\Hm{d}{n}$ on $\gamma_{d}$ by $P$. Assume to the contrary that there exists a subset $Q\subseteq P$ of  $d+2$ points $Q=\{q_1,\ldots q_{d+2}\}$, and hyperedges $A$ and $B$ such that the odd indexed points of $Q$ (along $\gamma_{d}$) $q_1,q_3,\ldots,$ belong to $A\setminus B$  
and the even indexed points $q_2,q_4,\ldots$ belong to $B \setminus A$. In particular $q_1,q_3,\ldots,$ belong to $A$ and  $q_2,q_4,\ldots$ do not belong to $A$. Let $h$ be a hyperplane such that the halfspace $h^+$ bounded by $h$ witness the hyperedge $A$, namely, $h^+ \cap P = A$. This means that the odd indexed points $q_1,q_3,\ldots,$  all in $h^+$ and  the even indexed points $q_2,q_4,\ldots$ lie in the complement halfspace $h^-$. In particular the hyperplane $h$ separates the even indexed points from the odd indexed points and therefore must intersect $\gamma_{d}$ between any two consecutive points $q_i$ and $q_{i+1}$ for $ i \in \{1,\ldots,d+1\}$. So $h$ intersects $\gamma_{d}$ in $d+1$ points. This is a contradiction to the fact that any hypeplane in $\mathbb{R}^d$ intersects $\gamma_d$ at most $d$ times.

On the other hand, the other direction does not hold in general. That is, $\II^m_n$ (and $\Hm{d}{n}$, respectively) and its sub-hypergraphs are just special cases 
of $(AB)^{(2m+1)/2}$-free (and $(AB)^{(d+2)/2}$-free, respectively) hypergraphs, see Figure \ref{fig:ABABA}. Hence it is natural to ask the following general question.

\begin{figure}
    \centering
    \includegraphics[width=0.5\linewidth]{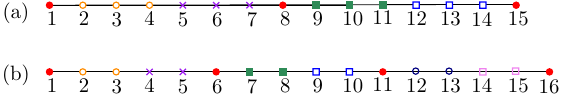}
    \caption{(a) An $ABABA$-free ($(AB)^{2.5}$-free) hypergraph which is not a $2$-interval hypergraph. (b) $ABABAB$-free ($(AB)^{3}$-free) hypergraph which is not a sub-hypergraph of $\Hm{4}{16}$ where $\gamma_4$ is described simply as a straight line. Each color describes distinct hyperedges. A similar construction can be obtained for more intervals or higher dimensions.}
    \label{fig:ABABA}
\end{figure}

\begin{Q} \label{question_AB-free}
For a positive integer $l$, what is the CF-chromatic number of an $(AB)^{l/2}$-free hypergraph without any hyperedges of size at most $\lfloor (l-1)/2 \rfloor$?    
\end{Q}
 Note that the threshold $\lfloor (l-1)/2 \rfloor$ comes from the threshold $d/2$ of Question \ref{question_main} by considering $\II^{(l-1)/2}_n$ (as a sub-hypergraph of $\Hm{l-1}{n}$) when $l$ is odd, and $\Hm{l-2}{n}$ when $l$ is even.

\smallskip

\textbf{Conflict-free coloring of general simplicial spheres.} 
Recently there have been studies on proper colorings and transversals of (the facet hypergraphs of) simplicial spheres \cite{transversals_spheres_joseph_michael,novik_transversal_open_problem, cho2023transversal, lee2023colorings, novik2024transversalnumberssimplicialpolytopes}. Many results were on lower bound constructions to obtain large chromatic number or transversal ratio (the ratio between transversal number and the vertex set size). There the moment curve or cyclic polytopes played an important role. As a side effect, the current lower bounds are relatively weaker for even-dimensional spheres than odd-dimensional ones as we could see for $\FC_d(n)$ in Proposition \ref{prop:proper_cyclic}.

On the other hand, in \cite{lee2023colorings} some other construction was obtained using different methods, which shows that for any dimension $d \geq 4$ and any positive integer $c$, there is a $d$-dimensional simplicial sphere whose facet hypergraph has a proper chromatic number at least $c$. Since we have $\CF(H)\geq \chi(H)$ for any hypergraph $H$, this implies that the CF-chromatic number is unbounded for even-dimensional spheres. Unfortunately, the construction highly depends on facts from PL topology so it does not have an explicit description to be used practically. This motivates the following question.

\begin{Q}
What is the asymptotically maximum CF-chromatic number of $d$-dimensional simplicial spheres for each even dimensions $d$? 
\end{Q}
\section*{Acknowledgements}
We want to thank Zolt\'{a}n F\"{u}redi, Dong Yeap Kang and Felix Lazebnik for pointing out relevant references in extremal graph theory. We also thank Dennis Wong and Lonc Zbigniew   for their answers to our questions on  universal cycles for subsets. Thanks are further extended to an anonymous refree who provided many helpful remarks.

\bibliographystyle{hplain}
\bibliography{bibliography}

\begin{thebibliography}{10}

\bibitem{AADFGHKS17}
Z.~Abel, V.~Alvarez, E.D. Demaine, S.~P. Fekete, A.~Gour, A.~Hesterberg,
  P.~Keldenich, and C.~Scheffer.
\newblock Three colors suffice: Conflict-free coloring of planar graphs.
\newblock In {\em Proc. 28th Annu. ACM-SIAM Sympos. Discrete Algorithms}, pages
  1951--1963, 2017.

\bibitem{ABAB_stabbed_pseudodisk}
Eyal Ackerman, Bal\'{a}zs Keszegh, and D\"{o}m\"{o}t\"{o}r P\'{a}lv\"{o}lgyi.
\newblock Coloring hypergraphs defined by stabbed pseudo-disks and {ABAB}-free
  hypergraphs.
\newblock {\em SIAM J. Discrete Math.}, 34(4):2250--2269, 2020.

\bibitem{AS08}
Noga Alon and Shakhar Smorodinsky.
\newblock Conflict-free colorings of shallow discs.
\newblock {\em Int. J. Comput. Geometry Appl.}, 18(6):599--604, 2008.

\bibitem{benson_c6_c10}
Clark~T. Benson.
\newblock Minimal regular graphs of girths eight and twelve.
\newblock {\em Canadian J. Math.}, 18:1091--1094, 1966.

\bibitem{universal_approx_blackburn}
Simon~R. Blackburn.
\newblock The existence of {$k$}-radius sequences.
\newblock {\em J. Combin. Theory Ser. A}, 119(1):212--217, 2012.

\bibitem{even_cycle_thm_bondy_simonovits}
J.~A. Bondy and M.~Simonovits.
\newblock Cycles of even length in graphs.
\newblock {\em J. Combinatorial Theory Ser. B}, 16:97--105, 1974.

\bibitem{transversals_spheres_joseph_michael}
Joseph Briggs, Michael~Gene Dobbins, and Seunghun Lee.
\newblock Transversals and colorings of simplicial spheres.
\newblock {\em Discrete Comput. Geom.}, 71(2):738--763, 2024, arXiv:2111.06560.

\bibitem{CheilarisGRS14}
Panagiotis Cheilaris, Luisa Gargano, Adele~A. Rescigno, and Shakhar
  Smorodinsky.
\newblock Strong conflict-free coloring for intervals.
\newblock {\em Algorithmica}, 70(4):732--749, 2014.

\bibitem{cf7}
K.~Chen, A.~Fiat, M.~Levy, J.~Matou{\v s}ek, E.~Mossel, J.~Pach, M.~Sharir,
  S.~Smorodinsky, U.~Wagner, and E.~Welzl.
\newblock Online conflict-free coloring for intervals.
\newblock {\em SIAM J. Comput.}, 36:545--554, 2006.

\bibitem{cf6}
Ke~Chen, Haim Kaplan, and Micha Sharir.
\newblock Online conflict-free coloring for halfplanes, congruent disks, and
  axis-parallel rectangles.
\newblock {\em ACM Transactions on Algorithms}, 5(2), 2009.

\bibitem{cho2023transversal}
Minho Cho and Jinha Kim.
\newblock Transversal numbers of stacked spheres.
\newblock {\em Discrete Math.}, 347(7):Paper No. 114061, 10, 2024.

\bibitem{universal_cycle_Chung_Diaconis_Graham}
Fan Chung, Persi Diaconis, and Ron Graham.
\newblock Universal cycles for combinatorial structures.
\newblock {\em Discrete Math.}, 110(1-3):43--59, 1992.

\bibitem{near_universal_cycle_subsets_Curtis}
Dawn Curtis, Taylor Hines, Glenn Hurlbert, and Tatiana Moyer.
\newblock Near-universal cycles for subsets exist.
\newblock {\em SIAM J. Discrete Math.}, 23(3):1441--1449, 2009.

\bibitem{complexity_ABAB}
Gábor Damásdi, Balázs Keszegh, Dömötör Pálvölgyi, and Karamjeet Singh.
\newblock The complexity of recognizing {$ABAB$}-free hypergraphs, 2025,
  arXiv:2409.01680.

\bibitem{near_universal_cycle_subsets_Lonc}
Micha\l{} D\c{e}bski and Zbigniew Lonc.
\newblock Universal cycle packings and coverings for {$k$}-subsets of an
  {$n$}-set.
\newblock {\em Graphs Combin.}, 32(6):2323--2337, 2016.

\bibitem{erdos_girth_conjecture}
P.~Erd\H{o}s.
\newblock Extremal problems in graph theory.
\newblock In {\em Theory of {G}raphs and its {A}pplications ({P}roc. {S}ympos.
  {S}molenice, 1963)}, pages 29--36. Publ. House Czech. Acad. Sci., Prague,
  1964.

\bibitem{ELRS}
G.~Even, Z.~Lotker, D.~Ron, and S.~Smorodinsky.
\newblock Conflict-free colorings of simple geometric regions with applications
  to frequency assignment in cellular networks.
\newblock {\em SIAM J. Comput.}, 33:94--136, 2003.

\bibitem{FeketeK18}
S.~P. Fekete and P.~Keldenich.
\newblock Conflict-free coloring of intersection graphs.
\newblock {\em Int. J. Comput. Geometry Appl.}, 28(3), 2018.

\bibitem{furedi_simonovits_degenerate_extremal_survey}
Zolt\'{a}n F\"{u}redi and Mikl\'{o}s Simonovits.
\newblock The history of degenerate (bipartite) extremal graph problems.
\newblock In {\em Erd\"{o}s centennial}, volume~25 of {\em Bolyai Soc. Math.
  Stud.}, pages 169--264. J\'{a}nos Bolyai Math. Soc., Budapest, 2013.

\bibitem{euler_tours_hypergraphs_kuhn}
Stefan Glock, Felix Joos, Daniela K\"{u}hn, and Deryk Osthus.
\newblock Euler tours in hypergraphs.
\newblock {\em Combinatorica}, 40(5):679--690, 2020.

\bibitem{Hajnal2011}
Andr{\'a}s Hajnal, Istv{\'a}n Juh{\'a}sz, Lajos Soukup, and Zolt{\'a}n
  Szentmikl{\'o}ssy.
\newblock Conflict free colorings of (strongly) almost disjoint set-systems.
\newblock {\em Acta Mathematica Hungarica}, 131(3):230--274, 2011.

\bibitem{HS02}
Sariel Har-Peled and Shakhar Smorodinsky.
\newblock Conflict-free coloring of points and simple regions in the plane.
\newblock {\em Discrete {\&} Computational Geometry}, 34(1):47--70, 2005.

\bibitem{universal_cycle_exact_Hurlbert}
Glenn Hurlbert.
\newblock On universal cycles for {$k$}-subsets of an {$n$}-set.
\newblock {\em SIAM J. Discrete Math.}, 7(4):598--604, 1994.

\bibitem{universal_cycle_exact_Jackson}
B.~W. Jackson.
\newblock Universal cycles of {$k$}-subsets and {$k$}-permutations.
\newblock {\em Discrete Math.}, 117(1-3):141--150, 1993.

\bibitem{JartouxKSY24}
Bruno Jartoux, Chaya Keller, Shakhar Smorodinsky, and Yelena Yuditsky.
\newblock Conflict-free colouring of subsets.
\newblock {\em Discret. Comput. Geom.}, 72(2):876--891, 2024.

\bibitem{jukna_book}
Stasys Jukna.
\newblock {\em Extremal combinatorics}.
\newblock Texts in Theoretical Computer Science. An EATCS Series. Springer,
  Heidelberg, second edition, 2011.
\newblock With applications in computer science.

\bibitem{KellerRS21}
Chaya Keller, Alexandre Rok, and Shakhar Smorodinsky.
\newblock Conflict-free coloring of string graphs.
\newblock {\em Discret. Comput. Geom.}, 65(4):1337--1372, 2021.

\bibitem{KellerS20a}
Chaya Keller and Shakhar Smorodinsky.
\newblock Conflict-free coloring of intersection graphs of geometric objects.
\newblock {\em Discret. Comput. Geom.}, 64(3):916--941, 2020.

\bibitem{ABA}
Bal\'azs Keszegh and D\"om\"ot\"or P\'alv\"olgyi.
\newblock An abstract approach to polychromatic coloring: shallow hitting sets
  in {ABA}-free hypergraphs and pseudohalfplanes.
\newblock {\em J. Comput. Geom.}, 10(1):1--26, 2019.

\bibitem{dual_ABAB2024}
Balázs Keszegh and Dömötör Pálvölgyi.
\newblock On dual-{ABAB}-free and related hypergraphs, 2024, arXiv:2406.13321.

\bibitem{general_high_girth_CD_Lazebnik_Ustimenko_Woldar}
F.~Lazebnik, V.~A. Ustimenko, and A.~J. Woldar.
\newblock A new series of dense graphs of high girth.
\newblock {\em Bull. Amer. Math. Soc. (N.S.)}, 32(1):73--79, 1995.

\bibitem{number_component_turan_cycle}
Felix Lazebnik and Raymond Viglione.
\newblock On the connectivity of certain graphs of high girth.
\newblock {\em Discrete Math.}, 277(1-3):309--319, 2004.

\bibitem{lee2023colorings}
Seunghun Lee and Eran Nevo.
\newblock On colorings of hypergraphs embeddable in $\mathbb{R}^d$, 2023,
  arXiv:2307.14195.

\bibitem{generalized_theta_liu_yang}
Xiao-Chuan Liu and Xu~Yang.
\newblock On the {T}ur\'{a}n number of generalized theta graphs.
\newblock {\em SIAM J. Discrete Math.}, 37(2):1237--1251, 2023.

\bibitem{near_universal_cycle_subsets_Lonc2}
Zbigniew Lonc, Tomasz Traczyk, and Miros\l~aw Truszczy\'nski.
\newblock Optimal {$f$}-graphs for the family of all {$k$}-subsets of an
  {$n$}-set.
\newblock In {\em Data base file organization ({W}arsaw, 1981)}, volume~6 of
  {\em Notes Rep. Comput. Sci. Appl. Math.}, pages 247--270. Academic Press,
  New York, 1983.

\bibitem{walecki}
Edouard Lucas.
\newblock {\em R\'{e}cr\'{e}ations math\'{e}matiques}.
\newblock Librairie Scientifique et Technique Albert Blanchard, Paris, 1960.
\newblock 2i\`eme \'{e}d., nouveau tirage.

\bibitem{upper_bound_mcmullen}
P.~McMullen.
\newblock The maximum numbers of faces of a convex polytope.
\newblock {\em Mathematika}, 17:179--184, 1970.

\bibitem{novik2024transversalnumberssimplicialpolytopes}
Isabella Novik and Hailun Zheng.
\newblock Transversal numbers of simplicial polytopes, spheres, and pure
  complexes, 2024, arXiv:2407.19693.

\bibitem{novik_transversal_open_problem}
Isabella Novik and Hailun Zheng.
\newblock Neighborly spheres and transversal numbers.
\newblock In {\em Open problems in algebraic combinatorics}, volume 110 of {\em
  Proc. Sympos. Pure Math.}, pages 113--128. Amer. Math. Soc., Providence, RI,
  [2024] \copyright 2024.
\newblock Available at
  http://www.samuelfhopkins.com/OPAC/files/proceedings/novik\_zheng.pdf.

\bibitem{CFPT09}
J.~Pach and G.~Tardos.
\newblock Conflict-free colourings of graphs and hypergraphs.
\newblock {\em Combinatorics, Probability {\&} Computing}, 18(5):819--834,
  2009.

\bibitem{smoro}
S.~Smorodinsky.
\newblock On the chromatic number of geometric hypergraphs.
\newblock {\em SIAM J. Discrete Mathematics}, 21:676--687, 2007.

\bibitem{CF-survey}
S.~Smorodinsky.
\newblock {\em Conflict-Free Coloring and its Applications, Geometry ---
  Intuitive, Discrete, and Convex}, pages 331--389.
\newblock Springer Berlin Heidelberg, Berlin, Heidelberg, 2013.

\bibitem{SmPHD}
Shakhar Smorodinsky.
\newblock {\em Combinatorial Problems in Computational Geometry}.
\newblock PhD thesis, School of Computer Science, Tel-Aviv University, 2003.

\bibitem{lectures_on_polytopes_book}
G\"{u}nter~M. Ziegler.
\newblock {\em Lectures on polytopes}, volume 152 of {\em Graduate Texts in
  Mathematics}.
\newblock Springer-Verlag, New York, 1995.

\end{thebibliography}

\end{document}